\documentclass{article}
\usepackage{a4wide}
\usepackage{amsmath}
\usepackage{amssymb}
\usepackage{amsthm}
\usepackage{marginnote}
\usepackage{dsfont}

\newcommand{\1}{\mathds{1}}

\newtheorem{thm}{Theorem}

\newtheorem{remark}{Remark}

\newtheorem{lemma}{Lemma}

\newtheorem*{prop*}{Proposition}

\newtheorem*{thm*}{Theorem}

\newcommand{\bbr}{{\mathbb R}}
\newcommand{\bbs}{{\mathbb S}}
\newcommand{\bbn}{{\mathbb N}}

\def\be{\begin{equation}}
\def\ee{\end{equation}}

\newcommand{\bea}{\begin{eqnarray}}

\newcommand{\eea}{\end{eqnarray}}
\newcommand{\bean}{\begin{eqnarray*}}

\newcommand{\eean}{\end{eqnarray*}}

\newcounter{mnotecount}[section]

\renewcommand{\themnotecount}{\thesection.\arabic{mnotecount}}

\newcommand{\mnote}[1]
{\protect{\stepcounter{mnotecount}}$^{\mbox{\footnotesize
$
\bullet$\themnotecount}}$ \marginpar{
\raggedright\tiny\em
$\!\!\!\!\!\!\,\bullet$\themnotecount: #1} }

\usepackage{authblk}

\begin{document}
\title{The massless Boltzmann equation in Minkowski spacetime}

\author[1]{Ho Lee\footnote{holee@khu.ac.kr}}
\author[2]{Ernesto Nungesser\footnote{em.nungesser@upm.es}}
\author[3]{John Stalker\footnote{stalker@maths.tcd.ie}}
\author[4]{Paul Tod\footnote{tod@maths.ox.ac.uk}}

\affil[1]{Department of Mathematics and Research Institute for Basic Science, College of Sciences, Kyung Hee University, Seoul 02447, Republic of Korea}
\affil[2]{M2ASAI, Universidad Polit{\'e}cnica de Madrid, ETSI Navales, Avda. de la Memoria
4, 28040 Madrid, Spain}
\affil[3]{School of Mathematics, Trinity College, Dublin 2, Ireland and Hamilton Mathematics Institute, Dublin 2, Ireland}
\affil[4]{Mathematical Institute, University of Oxford, Oxford OX2 6GG}

\maketitle

\begin{abstract}
We study the spatially homogeneous, massless Boltzmann equation in Minkowski spacetime for a certain range of hard and soft interactions. For hard interactions, we derive a Povzner-type inequality for massless particles and show that solutions exist for all time into the future. For soft interactions, we employ singular weights to control singularities at $ p = 0 $, which arise from the masslessness of particles, to obtain local existence. These results, which are among rather few proofs of existence for the massless Boltzmann equation, are motivated by our earlier work on the massless Einstein--Boltzmann system in certain cosmological settings.
\end{abstract}


%
\section{Introduction}

When considering partial differential equations it is natural to investigate the existence of local solutions and determine whether they can be extended globally. In this paper, we are interested in relativistic kinetic theory and will study the existence of local or global solutions to the \emph{massless} Boltzmann equation. For a comprehensive background on relativistic (and also non-relativistic) kinetic theory, we refer to \cite{Andreasson,CIP, CK, Ehlers,EMV, GLW, Stewart}.

For massive particles, the local existence of solutions to the relativistic Boltzmann equation, either on a fixed background or coupled to the Einstein equations, was considered in \cite{Bancel, BC}. These results have been generalised to show global existence in certain cases (see, e.g., \cite{GS, Lee2, LLN, LN, LN1, LN3, LN4, LN5, NDT, NT, SY, TAD}). To the best of our knowledge, the only existing results for \emph{massless} particles, apart from our recent work \cite{Lee21, LNST1,LNST2,LNT1}, are those in \cite{BDHMN, Bichteler, STR}. In \cite{Bichteler}, the local existence of solutions to the relativistic Boltzmann equation on a fixed background was obtained for both massless and massive particles. The fixed background is allowed to be quite general and no symmetries are assumed, but the collision cross-section $\sigma$ is assumed to be bounded (see (10) of \cite{Bichteler}). In \cite{BDHMN}, a particular global solution to the massless Boltzmann equation was constructed in a fixed cosmological background with a constant cross-section. Very recently in \cite{STR}, the future global stability of Maxwell--J{\"u}ttner equilibria and vacuum solutions to the massless Boltzmann equation was studied with a constant cross-section.

In this paper, we study the massless Boltzmann equation on a fixed Minkowski background but will allow for a wider range of collision cross-sections:
\begin{align*}
\sigma = \varrho^\gamma , \qquad - 2 < \gamma \leq 1 ,
\end{align*}
where $ \varrho $ denotes the relative momentum of two colliding particles (see \eqref{varrho} for its definition), and which contains both hard ($\gamma \geq 0 $) and soft ($\gamma < 0 $) interactions. The case of a constant cross-section corresponds to $\gamma=0$.

Our interest in massless particles comes from the fact that in cosmology, near the initial singularity, massive particles are expected to behave like massless ones in the sense that the ratio of the pressure to the energy density tends to $ 1 / 3 $, which corresponds to the massless case. Also, as a general rule, rest mass can be neglected if it is small compared to energy. However, existing literature on the massless Boltzmann equation is surprisingly scarce. There are several reasons for this scarcity. First, most available literature concerns the non-relativistic Boltzmann equation, and it is unnatural to consider massless particles. Second, the energy of massive particles is bounded away from zero, but this boundedness does not hold for massless particles. This introduces the extra complication that the energy of a massless particle can be arbitrarily small, which introduces a singularity in the collision term at $ p = 0 $. Also the distribution function is now supported on the light cone, which is not a smooth submanifold of the spacetime cotangent bundle, rather than on the mass  hyperboloid, which is. For these reasons, existence proofs for the massless Boltzmann or Einstein--Boltzmann system call for a new start -- one cannot just carry over proofs from the massive case.

In the present paper we seek to fill that gap by proving local existence for the massless Boltzmann equation for both soft and hard interactions. In the latter case we are able to prove global existence. It is important to note that our results should not to be confused with results concerning the stability of Minkowski spacetime, since here our background, Minkowski spacetime, is fixed and not coupled to the Einstein equations. In other words we are considering only the special relativistic case.

The paper is organised as follows. Below, we introduce the massless Boltzmann equation in Section \ref{sec MB} and state our main results in Section \ref{sec main}. In Section \ref{sec BE}, we present the basic estimates. In Section \ref{sec PM}, we prove the main results. We consider the modified Boltzmann equations \eqref{Bmod} and \eqref{Bmods}, where the collision cross-sections are bounded. We obtain the existence of solutions to \eqref{Bmod} and \eqref{Bmods} in Sections \ref{sec ECO} and \ref{sec ECOs}, respectively. In Sections \ref{sec E} and \ref{sec Es}, we prove the main results. We obtain global existence for hard interactions in Theorem \ref{thm hard} and local existence for soft interactions in Theorem \ref{thm soft}.

\subsection{The massless Boltzmann equation} \label{sec MB}
Let us consider the relativistic Boltzmann equation with massless particles. By the mass shell condition the massless particles with momentum $ p_\alpha $ satisfy
\begin{align*}
p_\alpha p^\alpha = - m^2 = 0 ,
\end{align*}
so that we obtain
\begin{align}
p^0 = \sqrt{ m^2 + | p |^2 } = | p | , \label{p^0}
\end{align}
where $ p $ denotes the three-dimensional vector $ ( p_1 , p_2 , p_3 ) $, and $ | \cdot | $ is the usual Euclidean norm. The relative momentum $ \varrho \geq 0 $ of two colliding particles with momenta $ p_\alpha $ and $ q_\alpha $ is defined by
\begin{align}\label{varrho}
\varrho^2 : = ( p_\alpha - q_\alpha ) ( p^\alpha - q^\alpha ) = 2 ( p^0 q^0 - p \cdot q ) .
\end{align}
In the case of massive particles we need to introduce a quantity $ s $, called the total energy, but in the massless case we have $ \varrho = \sqrt{ s } $. We assume that the collision cross-section does not depend on the scattering angle, and the spatially homogeneous, massless Boltzmann equation reads
\begin{align}\label{B}
\frac{ \partial f }{ \partial t } = \int_{ \bbr^3 } \int_{ \bbs^2 } \frac{ \varrho^{ 2 + \gamma } }{ p^0 q^0 } ( f ( p' ) f ( q' ) - f ( p ) f ( q ) ) \, d \omega \, d^3 q .
\end{align}
We will consider the case of both hard and soft interactions:
\begin{align*}
- 2 < \gamma \leq 1 ,
\end{align*}
but the analysis will be separated into the cases of hard $ ( \gamma \geq 0 ) $ and soft $ ( \gamma < 0 ) $ interactions. For the post-collision momentum we write
\begin{align*}
n^0 = p^0 + q^0 , \qquad n = p + q ,
\end{align*}
to have
\begin{align}
p'^0 & = \frac{ n^0 }{ 2 } + \frac{ n \cdot \omega }{ 2 } , \label{p'^0} \\
q'^0 & = \frac{ n^0 }{ 2 } - \frac{ n \cdot \omega }{ 2 } , \label{q'^0}
\end{align}
and
\begin{align}
p' & = \frac{ n }{ 2 } + \frac{ \varrho }{ 2 } \left( \omega + \frac{ ( n \cdot \omega ) n }{ \varrho ( n^0 + \varrho ) } \right) , \label{p'} \\
q' & = \frac{ n }{ 2 } - \frac{ \varrho }{ 2 } \left( \omega + \frac{ ( n \cdot \omega ) n }{ \varrho ( n^0 + \varrho ) } \right) , \label{q'}
\end{align}
for $ \omega \in \bbs^2 $. The above parametrization was used in \cite{LNT1}, which can be derived from \cite{SY} by taking $ \sqrt{ s } = \varrho $. By energy conservation we have
\begin{align}\label{energy}
n'^0 = n^0 ,
\end{align}
where $ n'^0 = p'^0 + q'^0 $, and the change of variables between $ ( p , q ) $ and $ ( p' , q' ) $ implies
\begin{align}\label{dpdq}
\frac{ 1 }{ p^0 q^0 } \, d \omega \, d^3 p \, d^3 q = \frac{ 1 }{ p'^0 q'^0 } \, d \omega \, d^3 p' \, d^3 q' .
\end{align}
The following are well-known facts for the relativistic Boltzmann equation
\begin{align}
\frac{ d }{ d t } \int_{ \bbr^3 } f \, d^3 p = 0 , \qquad \frac{ d }{ d t } \int_{ \bbr^3 } f p^0 d^3 p = 0 , \label{conserved}
\end{align}
which hold also for the massless Boltzmann equation.

\subsection{Main results}\label{sec main}
In this part we state our main results. We define the following weighted $ L^1 $ and $ L^\infty $ norms:
\begin{align*}
\| f \|_{ L^1_r } = \int f ( p ) ( p^0 )^r d^3 p , \qquad \| f \|_{ L^\infty_w } = \sup_{ p \in \bbr^3 } | w ( p ) f ( p ) | , \qquad w ( p ) = p^0 e^{ p^0 } .
\end{align*}
Then, we obtain the following results. In Theorem \ref{thm hard}, we obtain global existence for hard interactions. In Theorem \ref{thm soft}, we only obtain local existence for soft interactions. The proof will be given in Sections \ref{sec E} and \ref{sec Es}, respectively.

\begin{thm}
Let $ 0 \leq f_0 \in L^1 ( \bbr^3 ) \cap L^1_2 ( \bbr^3 ) $ be an initial datum of the Boltzmann equation \eqref{B} with $ 0 \leq \gamma \leq 1 $. Then, there exists a unique non-negative solution $ f \in C^1 ( [ 0 , \infty ) ; L^1 ( \bbr^3 ) \cap L^1_1 ( \bbr^3 ) ) $. If, in addition, $ f_0 \in L^1_r ( \bbr^3 ) $ with $ r \geq 1 $, then for any $ T > 0 $,
\begin{align}
\sup_{ t \in [ 0 , T ] } \| f ( t ) \|_{ L^1_r } \leq C , 
\end{align}
where $ C > 0 $ depends on $ T $ and $r$.
\end{thm}

\begin{thm}
Let $ 0 \leq f_0 \in L^1 ( \bbr^3 ) \cap L^1_{ - 2 } ( \bbr^3 ) \cap L^\infty_w ( \bbr^3 ) $ be an initial datum of the Boltzmann equation \eqref{B} with $ - 2 < \gamma < 0 $. Then, there exists a $ T > 0 $ such that the Boltzmann equation \eqref{B} has a unique non-negative solution $ f \in C^1 ( [ 0 , T ] ; L^1 ( \bbr^3 ) \cap L^1_{ - 1 } ( \bbr^3 ) ) $ satisfying
\begin{align}
\sup_{ t \in [ 0 , T ] } ( \| f ( t ) \|_{ L^1_{ - 2 } } + \| f ( t ) \|_{ L^\infty_w } ) \leq C , 
\end{align}
where $ C > 0 $ depends on $ T $.
\end{thm}

\begin{remark}
In \cite{Lee21}, massless particles were considered in a fixed homogeneous and isotropic cosmological background with a scale factor $ R \sim t^{ \frac12 } $. The main argument was to use the monotonicity and integrability of $R^{-3-\gamma}$, which led to global existence for $ - 1 < \gamma < 2 $. However, this argument is no longer valid in the Minkowski case, which would correspond to $R=1$. For hard interactions, we instead use a Povzner-type inequality to obtain global existence for $ 0 \leq \gamma < 1 $. For soft interactions, we only obtain local existence, but for a wider range of collision cross-sections: $ - 2 < \gamma < 0 $.
\end{remark}

\setcounter{thm}{0}

\section{Basic estimates}\label{sec BE}

We collect basic estimates. Lemma \ref{lem h} gives a different expression for the relative momentum of massless particles. Lemma \ref{lem p'} is a Povzner-type inequality for massless particles. This will be used to obtain global existence for hard interactions. Such a Povzner-type inequality is unavailable for soft interactions, but the estimates in Lemma \ref{lem 1/p'} will be used to control the singularity of the relative momentum.

\begin{lemma}\label{lem h}
The relative momentum $ \varrho $ is a collisional invariant and can be written as
\begin{align}
\varrho = \sqrt{ ( n^0 )^2 - | n |^2 } = 2 \sqrt{ p^0 q^0 } \sin \frac{ \phi }{ 2 } , \label{l1}
\end{align}
where $ \phi $ is the angle between the three-dimensional vectors $ p $ and $ q $.
\end{lemma}
\begin{proof}
We refer to \cite{LNT1} for the proof.
\end{proof}

\begin{lemma}\label{lem p'}
Let $ p'^0 $ and $ q'^0 $ be the post-collision momenta for given $ p , q \in \bbr^3 $ and $ \omega \in \bbs^2 $. Then, we have for $ r \geq 1 $,
\begin{align}
( p'^0 )^r + ( q'^0 )^r - ( p^0 )^r - ( q^0 )^r \leq C ( ( p^0 )^{ r - 1 } q^0 + p^0 ( q^0 )^{ r - 1 } ) ,\label{l2}
\end{align}
for some $ C > 0 $ which depends on $r$.
Moreover, there exists $ d_r > 0 $ such that if $ \omega $ is restricted to $ | n \cdot \omega | \leq d_r | n | $, then we have
\begin{align}
\nonumber & ( p'^0 )^r + ( q'^0 )^r - ( p^0 )^r - ( q^0 )^r \\
& \leq C ( ( p^0 )^{ r - 1 } q^0 + p^0 ( q^0 )^{ r - 1 } ) - \left( 1 - \frac{ 2 ( 1 + d_r )^r }{ 2^r } \right) ( ( p^0 )^r + ( q^0 )^r ) , \label{l3}
\end{align}
where $ 1 - 2 ( 1 + d_r )^r / 2^r > 0 $.
\end{lemma}
\begin{proof}
The proof is almost the same as the one in \cite{LR}. We suppose that $ p $ and $ q $ are given and notice that the left hand side of \eqref{l2} attains its maximum when $ \omega $ is parallel to $ n $. Hence, we have
\begin{align*}
( p'^0 )^r + ( q'^0 )^r - ( p^0 )^r - ( q^0 )^r & = \left( \frac{ n^0 }{ 2 } + \frac{ n \cdot \omega }{ 2 } \right)^r + \left( \frac{ n^0 }{ 2 } - \frac{ n \cdot \omega }{ 2 } \right)^r - ( p^0 )^r - ( q^0 )^r \\
& \leq \left( \frac{ n^0 + | n | }{ 2 } \right)^r + \left( \frac{ n^0 - | n | }{ 2 } \right)^r - ( p^0 )^r - ( q^0 )^r \\
& \leq \left( \frac{ n^0 + | n | }{ 2 } \right)^r - \max \Big( ( p^0 )^r , ( q^0 )^r \Big) ,
\end{align*}
where we used the fact that $ ( n^0 - | n | ) / 2 $ is smaller than both $ p^0 $ and $ q^0 $. Since $ ( n^0 + | n | ) / 2 \leq n^0 $, we have
\begin{align*}
( p'^0 )^r + ( q'^0 )^r - ( p^0 )^r - ( q^0 )^r & \leq ( p^0 + q^0 )^r - \max \Big( ( p^0 )^r , ( q^0 )^r \Big) \\
& \leq ( p^0 )^r + ( q^0 )^r + C_r ( ( p^0 )^{ r - 1 } q^0 + p^0 ( q^0 )^{ r - 1 } ) - \max \Big( ( p^0 )^r , ( q^0 )^r \Big) \\
& \leq C_r ( ( p^0 )^{ r - 1 } q^0 + p^0 ( q^0 )^{ r - 1 } ) ,
\end{align*}
for some $ C_r > 0 $. For the second result we notice that
\begin{align*}
p'^0 & \leq \frac{ n^0 }{ 2 } + \frac{ | n \cdot \omega | }{ 2 } \leq \frac{ n^0 }{ 2 } + \frac{ d_r | n | }{ 2 } \leq \frac{ ( 1 + d_r ) n^0 }{ 2 } ,
\end{align*}
and have the same estimate for $ q'^0 $. Now, we obtain
\begin{align*}
& ( p'^0 )^r + ( q'^0 )^r - ( p^0 )^r - ( q^0 )^r \\
& \leq \frac{ 2 ( 1 + d_r )^r ( n^0 )^r }{ 2^r } - ( p^0 )^r - ( q^0 )^r \\
& \leq \frac{ 2 ( 1 + d_r )^r }{ 2^r } ( ( p^0 )^r + ( q^0 )^r + C_r ( ( p^0 )^{ r - 1 } q^0 + p^0 ( q^0 )^{ r - 1 } ) ) - ( p^0 )^r - ( q^0 )^r \\
& \leq C_r ( ( p^0 )^{ r - 1 } q^0 + p^0 ( q^0 )^{ r - 1 } ) - \left( 1 - \frac{ 2 ( 1 + d_r )^r }{ 2^r } \right) ( ( p^0 )^r + ( q^0 )^r ) .
\end{align*}
Now, we can choose a small $ d_r $ so that the factor $ ( 1 - 2 ( 1 + d_r )^r / 2^r ) $ is strictly positive.
\end{proof}

In the case of $ r= 2 $ one can obtain the result of Lemma \ref{lem p'} by direct calculation:
\begin{align}
( p'^0 )^2 + ( q'^0 )^2 - ( p^0 )^2 - ( q^0 )^2 \leq 2 p^0 q^0 ,
\end{align}
as in \cite{LNT1}.

\begin{lemma}\label{lem 1/p'}
The post-collision momenta satisfy for any $ a > 0 $ the following:
\begin{align*}
& \int_{ \bbs^2 } \frac{ 1 }{ p'^0 } \, d \omega = \int_{ \bbs^2 } \frac{ 1 }{ q'^0 } \, d \omega \leq \frac{ C }{ \varrho^a ( n^0 )^{ 1 - a } } , \\
& \int_{ \bbs^2 } \frac{ 1 }{ ( p'^0 )^2 } \, d \omega = \int_{ \bbs^2 } \frac{ 1 }{ ( q'^0 )^2 } \, d \omega = \frac{ 16 \pi }{ \varrho^2 } , \\
& \int_{ \bbs^2 } \frac{ 1 }{ p'^0 q'^0 } \, d \omega \leq \frac{ C }{ \varrho^a ( n^0 )^{ 2 - a } } ,
\end{align*}
where $ C > 0 $ depends on $ a $.
\end{lemma}
\begin{proof}
We refer to \cite{Lee21, LNT1} for the proof.
\end{proof}

\section{Proof of the main results}\label{sec PM}
In this paper we consider collision cross-sections of the type:
\begin{align*}
\sigma = \varrho^\gamma , \qquad - 2 < \gamma \leq 1 .
\end{align*}
They are not bounded for both hard and soft interactions. We follow the standard approach to cut off the collision cross-sections. For hard interactions, we consider
\begin{align}
\frac{ \partial f }{ \partial t } = Q^{\mathrm{hard}}_n ( f , f ) , \label{Bmod}
\end{align}
where $ Q_n^{\mathrm{hard}} $ is the collision operator with cut-off defined by 
\begin{align}
Q_n^{\mathrm{hard}} ( f , f ) = \iint_{ \varrho \leq n } \frac{ \varrho^{ 2 + \gamma } }{ p^0 q^0 } ( f ( p' ) f ( q' ) - f ( p ) f ( q ) ) \, d \omega \, d^3 q , \qquad 0 \leq \gamma \leq 1 . \label{Cmod}
\end{align} 
For soft interactions, we consider
\begin{align}
\frac{ \partial f }{ \partial t } = Q^{\mathrm{soft}}_n ( f , f ) , \label{Bmods}
\end{align}
with cut-off defined by
\begin{align}
Q_n^{\mathrm{soft}} ( f , f ) = \iint_{ \varrho \geq n^{ - 1 } } \frac{ \varrho^{ 2 + \gamma } }{ p^0 q^0 } ( f ( p' ) f ( q' ) - f ( p ) f ( q ) ) \, d \omega \, d^3 q , \qquad - 2 < \gamma < 0 . \label{Cmods}
\end{align} 
The collision kernels are now bounded, so that the modified equations \eqref{Bmod} with \eqref{Cmod} and \eqref{Bmods} with \eqref{Cmods} admit global solutions in $ L^1 $ by standard arguments (see, e.g., \cite{ark1, ark2, MW}). Below, in Sections \ref{sec ECO} and \ref{sec E}, we consider $ L^1_r $ with $ r > 0 $ to obtain global existence for hard interactions. In Sections \ref{sec ECOs} and \ref{sec Es}, we consider $ L^1_r  \cap L^\infty_w $ with $ r < 0 $ for soft interactions, and will obtain local existence.

\subsection{Existence for hard interactions with cut-off}\label{sec ECO}

The modified equation \eqref{Bmod} with \eqref{Cmod} has global, non-negative solutions in $ L^1 $. Let $ f_n $ be the solution of \eqref{Bmod} with \eqref{Cmod}. We need to show that the $ f_n $ are uniformly bounded with respect to $ n $. It is clear that the $ f_n $ are uniformly bounded in $ L^1 $ and $ L^1_1 $ as in \eqref{conserved}. We consider the space $ L^1_r $ with $ r \geq 1 $. We first multiply equation \eqref{Bmod} by $ ( p^0 )^{ 1 + \frac{ \gamma }{ 2 } } $ to obtain
\begin{align*}
& \frac{ d }{ d t } \int f_n ( p ) ( p^0 )^{ 1 + \frac{ \gamma }{ 2 } } \, d^3 p \\
& = \frac12 \iiint_{ \varrho \leq n } \frac{ \varrho^{ 2 + \gamma } }{ p^0 q^0 } f_n ( p ) f_n ( q ) ( ( p'^0 )^{ 1 + \frac{ \gamma }{ 2 } } + ( q'^0 )^{ 1 + \frac{ \gamma }{ 2 } } - ( p^0 )^{ 1 + \frac{ \gamma }{ 2 } } - ( q^0 )^{ 1 + \frac{ \gamma }{ 2 } } ) \, d \omega \, d^3 q \, d^3 p .
\end{align*}
Applying Lemma \ref{lem p'}, we obtain
\begin{align*}
\frac{ d }{ d t } \int f_n ( p ) ( p^0 )^{ 1 + \frac{ \gamma }{ 2 } } \, d^3 p & \leq C \iiint_{ \varrho \leq n } \frac{ \varrho^{ 2 + \gamma } }{ p^0 q^0 } f_n ( p ) f_n ( q ) ( p^0 )^{ \frac{ \gamma }{ 2 } } q^0 \, d \omega \, d^3 q \, d^3 p \\
& \leq C \iint_{ \bbr^6 } f_n ( p ) f_n ( q ) ( p^0 )^\gamma ( q^0 )^{ 1 + \frac{ \gamma }{ 2 } } \, d^3 q \, d^3 p \\
& \leq C \| f_n \|_{ L^1_\gamma } \| f_n \|_{ L^1_{ 1 + \frac{ \gamma }{ 2 } } } .
\end{align*}
Since $ 0 \leq \gamma \leq 1 $, the quantity $ \| f_n \|_{ L^1_\gamma } $ is bounded. Hence, we obtain
\begin{align*}
\frac{ d }{ d t } \| f_n \|_{ L^1_{ 1 + \frac{ \gamma }{ 2 } } } \leq C \| f_n \|_{ L^1_{ 1 + \frac{ \gamma }{ 2 } } } ,
\end{align*}
which shows that the $ \| f_n \|_{ L^1_{ 1 + \frac{ \gamma }{ 2 } } } $ are uniformly bounded on any interval $ [ 0 , T ] $. Now, we multiply \eqref{Bmod} by $ ( p^0 )^r $ for $ r \geq 1 $ to obtain
\begin{align*}
\frac{ d }{ d t } \int f_n ( p ) ( p^0 )^r \, d^3 p & \leq C \iiint_{ \varrho \leq n } \frac{ \varrho^{ 2 + \gamma } }{ p^0 q^0 } f_n ( p ) f_n ( q ) ( p^0 )^{ r - 1 } q^0 \, d \omega \, d^3 q \, d^3 p \\
& \leq C \iint_{ \bbr^6 } f_n ( p ) f_n ( q ) ( p^0 )^{ r - 1 + \frac{ \gamma }{ 2 } } ( q^0 )^{ 1 + \frac{ \gamma }{ 2 } } \, d^3 q \, d^3 p \\
& \leq C \| f_n \|_{ L^1_{ r - 1 + \frac{ \gamma }{ 2 } } } \| f_n \|_{ L^1_{ 1 + \frac{ \gamma }{ 2 } } } .
\end{align*}
Since $ 0 \leq r - 1 + \frac{ \gamma }{ 2 } \leq r $, we conclude that the $ \| f_n \|_{ L^1_r } $ are uniformly bounded on any interval $ [ 0 , T ] $. Hence, we obtain the following lemma:

\begin{lemma}\label{lem f_n}
For any initial datum $ 0 \leq f_0 \in L^1 ( \bbr^3 ) $, equation \eqref{Bmod} for hard interactions with cut-off \eqref{Cmod} has a unique non-negative solution $ f_n \in C^1 ( [ 0 , \infty ) ; L^1 ( \bbr^3 ) ) $. If, in addition, $ f_0 \in L^1_r ( \bbr^3 ) $ with $ r \geq 0 $, then for any $ T > 0 $,
\begin{align*}
\sup_{ n \in \bbn } \sup_{ t \in [ 0 , T ] } \| f_n ( t ) \|_{ L^1_r } \leq C ,
\end{align*}
where $ C > 0 $ depends on $ T $ and $r$.
\end{lemma}

\subsection{Existence for soft interactions with cut-off}\label{sec ECOs}
As in Section \ref{sec ECO}, the modified equation \eqref{Bmods} with \eqref{Cmods} has global, non-negative solutions in $ L^1 $. In the case of soft interactions we need to show that the $ f_n $ are uniformly bounded in $ L^\infty_{ w } $ and $ L^1_{ - 2 } $ as well as in $ L^1 $. It is clear that the $ f_n $ are uniformly bounded in $ L^1 $. 

For $ L^\infty_w $, we multiply equation \eqref{Bmods} by $ w = p^0 e^{ p^0 } $ to obtain
\begin{align*}
\frac{ \partial ( w f_n ) }{ \partial t } & \leq C \| f \|^2_{ L^\infty_w } \iint_{ \varrho \geq n^{ - 1 } } \frac{ \varrho^{ 2 + \gamma } }{ p^0 q^0 } \frac{ w ( p ) }{ w ( p' ) w ( q' ) } \, d \omega \, d^3 q \\
& \leq C \| f \|^2_{ L^\infty_w } \iint \frac{ \varrho^{ 2 + \gamma } }{ q^0 } \frac{ e^{ - q^0 } }{ p'^0 q'^0 } \, d \omega \, d^3 q \\
& \leq C \| f \|^2_{ L^\infty_w } \int \frac{ \varrho^{ 2 + \gamma } }{ q^0 } \frac{ e^{ - q^0 } }{ \varrho^a ( n^0 )^{ 2 - a } } \, d^3 q ,
\end{align*}
where we used Lemma \ref{lem 1/p'}. We choose $ a = 2 + \gamma $, which is positive for any $ - 2 < \gamma < 0 $, to obtain
\begin{align*}
& \int \frac{ \varrho^{ 2 + \gamma } }{ q^0 } \frac{ e^{ - q^0 } }{ \varrho^a ( n^0 )^{ 2 - a } } \, d^3 q \leq C \int \frac{ e^{ - q^0 } }{ ( q^0 )^{ 1 + \gamma } } \, d^3 q ,
\end{align*}
which is finite for $ - 2 < \gamma < 0 $. We obtain
\begin{align*}
\frac{ d }{ d t } \| f_n \|_{ L^\infty_w } \leq C \| f_n \|_{ L^\infty_w }^2 ,
\end{align*}
which shows that $ f_n $ are bounded on a finite time interval $ [ 0 , T ] $. Next, for $ L^1_{ - 2 } $, we multiply equation \eqref{Bmods} by $ ( p^0 )^{ - 2 } $ to obtain
\begin{align*}
& \frac{ d }{ d t } \int f_n ( p ) \frac{ 1 }{ ( p^0 )^2 } \, d^3 p \\
& = \frac12 \iiint_{ \varrho \geq n^{ - 1 } } \frac{ \varrho^{ 2 + \gamma } }{ p^0 q^0 } f_n ( p ) f_n ( q ) \left( \frac{ 1 }{ ( p'^0 )^2 } + \frac{ 1 }{ ( q'^0 )^2 } - \frac{ 1 }{ ( p^0 )^2 } - \frac{ 1 }{ ( q^0 )^2 } \right) d \omega \, d^3 q \, d^3 p .
\end{align*}
Applying Lemma \ref{lem 1/p'} with the above estimate in $ L^\infty_w $, we obtain
\begin{align*}
\frac{ d }{ d t } \int f_n ( p ) \frac{ 1 }{ ( p^0 )^2 } \, d^3 p & \leq C \iint_{ \varrho \geq n^{ - 1 } } \frac{ \varrho^{ \gamma } }{ p^0 q^0 } f_n ( p ) f_n ( q ) \, d^3 q \, d^3 p \\
& \leq C \iint \frac{ \varrho^{ \gamma } }{ ( p^0 q^0 )^2 } e^{ - p^0 } e^{- q^0 } d^3 q \, d^3 p \\
& \leq C \iint \frac{ \sin^\gamma ( \frac{ \phi }{ 2 } ) }{ ( p^0 q^0 )^{ 2 - \frac{ \gamma }{ 2 } } } e^{ - p^0 } e^{- q^0 } d^3 q \, d^3 p ,
\end{align*}
where we used \eqref{l1}. The last double integration is finite, since $ - 2 < \gamma < 0 $. Hence, we conclude that the $ f_n $ are uniformly bounded in $ L^1_{ - 2 } $ on the time interval $ [ 0 , T ] $, which we obtained above. Hence, we obtain the following lemma:

\begin{lemma}\label{lem f_ns}
For any initial datum $ 0 \leq f_0 \in L^1 ( \bbr^3 ) $, the equation \eqref{Bmods} for soft interactions with cut-off \eqref{Cmods} has a unique non-negative solution $ f_n \in C^1 ( [ 0 , \infty ) ; L^1 ( \bbr^3 ) ) $. If, in addition, $ f_0 \in L^1_{ - 2 } ( \bbr^3 ) \cap L^\infty_w ( \bbr^3 ) $, then there exists a $ T > 0 $ such that
\begin{align*}
\sup_{ n \in \bbn } \sup_{ t \in [ 0 , T ] } ( \| f_n ( t ) \|_{ L^1_{ - 2 } } + \| f_n ( t ) \|_{ L^\infty_w } ) \leq C ,
\end{align*}
where $ C > 0 $ depends on $ T $.
\end{lemma}

\subsection{Proof of Theorem \ref{thm hard}}\label{sec E}

We restate Theorem \ref{thm hard} here for the reader's convenience.

\begin{thm}\label{thm hard}
Let $ 0 \leq f_0 \in L^1 ( \bbr^3 ) \cap L^1_2 ( \bbr^3 ) $ be an initial datum of the Boltzmann equation \eqref{B} with $ 0 \leq \gamma \leq 1 $. Then, there exists a unique non-negative solution $ f \in C^1 ( [ 0 , \infty ) ; L^1 ( \bbr^3 ) \cap L^1_1 ( \bbr^3 ) ) $. If, in addition, $ f_0 \in L^1_r ( \bbr^3 ) $ with $ r \geq 1 $, then for any $ T > 0 $,
\begin{align}
\sup_{ t \in [ 0 , T ] } \| f ( t ) \|_{ L^1_r } \leq C , \label{thm hard r}
\end{align}
where $ C > 0 $ depends on $ T $ and $r$.
\end{thm}
\begin{proof}
We need to show that the solutions $ f_n $ of Lemma \ref{lem f_n} converge to the solution of the original Boltzmann equation \eqref{B} as $ n \to \infty $. We will show that the sequence $ f_n $ converges in $ L^1 \cap L^1_1 $. Let $ k < n $, and consider
\begin{align*}
\frac{ \partial ( f_k - f_n ) }{ \partial t } & = \iint \1_{ \{ \varrho \leq k \} } \frac{ \varrho^{ 2 + \gamma } }{ p^0 q^0 } ( f_k ( p' ) f_k ( q' ) - f_k ( p ) f_k ( q ) ) \, d \omega \, d^3 q \\
& \quad - \iint \1_{ \{ \varrho \leq n \} } \frac{ \varrho^{ 2 + \gamma } }{ p^0 q^0 } ( f_n ( p' ) f_n ( q' ) - f_n ( p ) f_n ( q ) ) \, d \omega \, d^3 q ,
\end{align*}
where $ \1 $ denotes the indicator function. We write the above as
\begin{align*}
\frac{ \partial ( f_k - f_n ) }{ \partial t } & = - \iint \1_{ \{ k \leq \varrho \leq n \} } \frac{ \varrho^{ 2 + \gamma } }{ p^0 q^0 } ( f_k ( p' ) f_k ( q' ) - f_k ( p ) f_k ( q ) ) \, d \omega \, d^3 q \\
& \hspace{-1cm} + \iint \1_{ \{ \varrho \leq n \} } \frac{ \varrho^{ 2 + \gamma } }{ p^0 q^0 } ( f_k ( p' ) f_k ( q' ) - f_k ( p ) f_k ( q ) - f_n ( p' ) f_n ( q' ) + f_n ( p ) f_n ( q ) ) \, d \omega \, d^3 q .
\end{align*}
Multiplying the above by $ \mbox{sgn} ( f_k - f_n ) ( p ) $ and $ ( p^0 )^r $ with $ r = 0 , 1 $, and integrating over $ \bbr^3_p $, we obtain
\begin{align}
& \frac{ d }{ d t } \| f_k - f_n \|_{ L^1_r } \nonumber \\
& \leq \iiint \1_{ \{ k \leq \varrho \leq n \} } \frac{ \varrho^{ 2 + \gamma } }{ p^0 q^0 } ( f_k ( p' ) f_k ( q' ) + f_k ( p ) f_k ( q ) ) ( p^0 )^r d \omega \, d^3 q \, d^3 p \nonumber \\
& \quad + \frac12 \iiint \1_{ \{ \varrho \leq n \} } \frac{ \varrho^{ 2 + \gamma } }{ p^0 q^0 } \Big( | f_k - f_n | ( p' ) ( f_k + f_n ) ( q' ) + ( f_k + f_n ) ( p' ) | f_k - f_n | ( q' ) \nonumber \\
& \hspace{2cm} - | f_k - f_n | ( p ) ( f_k + f_n ) ( q ) + ( f_k + f_n ) ( p ) | f_k - f_n | ( q ) \Big) ( p^0 )^r d \omega \, d^3 q \, d^3 p \nonumber \\
& = \iiint \1_{ \{ k \leq \varrho \leq n \} } \frac{ \varrho^{ 2 + \gamma } }{ p^0 q^0 } f_k ( p ) f_k ( q ) ( ( p^0 )^r + ( p'^0 )^r ) \, d \omega \, d^3 q \, d^3 p \nonumber \\
& \quad + \frac12 \iiint \1_{ \{ \varrho \leq n \} } \frac{ \varrho^{ 2 + \gamma } }{ p^0 q^0 } ( f_k + f_n ) ( p ) | f_k - f_n | ( q ) ( ( p'^0 )^r + ( q'^0 )^r - ( q^0 )^r + ( p^0 )^r ) \, d \omega \, d^3 q \, d^3 p \nonumber \\
& \leq C \iiint \1_{ \{ k \leq \varrho \leq n \} } \frac{ \varrho^{ 2 + \gamma } }{ p^0 q^0 } f_k ( p ) f_k ( q ) ( p^0 )^r d \omega \, d^3 q \, d^3 p \label{est1} \\
& \quad + \iiint \1_{ \{ \varrho \leq n \} } \frac{ \varrho^{ 2 + \gamma } }{ p^0 q^0 } ( f_k + f_n ) ( p ) | f_k - f_n | ( q ) ( p^0 )^r d \omega \, d^3 q \, d^3 p , \label{est2}
\end{align}
where we used \eqref{energy} for $ r = 1 $. The quantity in \eqref{est1} can be estimated as
\begin{align*}
& \iiint \1_{ \{ k \leq \varrho \leq n \} } \frac{ \varrho^{ 2 + \gamma } }{ p^0 q^0 } f_k ( p ) f_k ( q ) ( p^0 )^r \, d \omega \, d^3 q \, d^3 p \\
& \leq \frac{ C }{ k } \iint \frac{ \varrho^{ 3 + \gamma } }{ p^0 q^0 } f_k ( p ) f_k ( q ) ( p^0 )^r \, d^3 q \, d^3 p \\
& \leq \frac{ C }{ k } \| f_k \|_{ L^1 _{ r + \frac{ 1 + \gamma }{ 2 } } } \| f_k \|_{ L^1_{ \frac{ 1 + \gamma }{ 2 } } } \\
& \leq \frac{ C }{ k } ,
\end{align*}
since the quantities $ \| f_k \|_{ L^1 _{ r + ( 1 + \gamma ) / 2 } } $ and $ \| f_k \|_{ L^1_{ ( 1 + \gamma ) / 2 } } $ are bounded by Lemma \ref{lem f_n} for $ r = 0 , 1 $. The quantity in \eqref{est2} is estimated as
\begin{align*}
& \iiint \1_{ \{ \varrho \leq n \} } \frac{ \varrho^{ 2 + \gamma } }{ p^0 q^0 } ( f_k + f_n ) ( p ) | f_k - f_n | ( q ) ( p^0 )^r d \omega \, d^3 q \, d^3 p \\
& \leq C \iint ( f_k + f_n ) ( p ) | f_k - f_n | ( q ) ( p^0 )^{ r + \frac{ \gamma }{ 2 } } ( q^0 )^{ \frac{ \gamma }{ 2 } } \, d^3 q \, d^3 p \\
& \leq C \| f_k + f_n \|_{ L^1_{ r + \frac{ \gamma }{ 2 } } } \| f_k - f_n \|_{ L^1_{ \frac{ \gamma }{ 2 } } } \\
& \leq C \| f_k - f_n \|_{ L^1_{ \frac{ \gamma }{ 2 } } } .
\end{align*}
The last quantity can be estimated as the sum of $ \| f_k - f_n \|_{ L^1 } $ and $ \| f_k - f_n \|_{ L^1_1 } $ since $ 0 \leq \gamma \leq 1 $. We combine the above results to obtain
\begin{align*}
\frac{ d }{ d t } \left( \| f_k - f_n \|_{ L^1 } + \| f_k - f_n \|_{ L^1_1 } \right) \leq C \left( k^{ - 1 } + \| f_k - f_n \|_{ L^1 } + \| f_k - f_n \|_{ L^1_1 } \right) .
\end{align*}
This shows that the sequence $ f_k $ converges in $ L^1 \cap L^1_1 $ as $ k \to \infty $ on each interval $ [ 0 , T ] $, so that the solution exists globally in time. The property \eqref{thm hard r} can be obtained as in Lemma \ref{lem f_n}, and this completes the proof.
\end{proof}

\subsection{Proof of Theorem \ref{thm soft}} \label{sec Es}
In this part we show that the solutions $ f_n $ of Lemma \ref{lem f_ns} converge to the solution of the original Boltzmann equation \eqref{B} for $ - 2 < \gamma < 0 $. In \cite{Lee21}, the massless Boltzmann equation was considered in a cosmological setting with a restricted range of the collision cross-section by $ \gamma > - 1 $. This restriction was necessary to obtain the integrability of $ R^{ - 3 - \gamma } $, which leads to global existence. In a different cosmological setting, the massless Boltzmann equation was studied in \cite{LNT1}, restricting the collision cross-section to $ \gamma < - 1 $ to avoid the singularity at $ t = 0 $. In this paper, for soft interactions, we will only consider local existence and have no singularity at $ t = 0 $, so the restrictions of \cite{Lee21, LNT1} do not apply to our case. We obtain local existence for $ - 2 < \gamma < 0 $ in Theorem \ref{thm soft}, which we restated below.

\begin{thm}\label{thm soft}
Let $ 0 \leq f_0 \in L^1 ( \bbr^3 ) \cap L^1_{ - 2 } ( \bbr^3 ) \cap L^\infty_w ( \bbr^3 ) $ be an initial datum of the Boltzmann equation \eqref{B} with $ - 2 < \gamma < 0 $. Then, there exists a $ T > 0 $ such that the Boltzmann equation \eqref{B} has a unique non-negative solution $ f \in C^1 ( [ 0 , T ] ; L^1 ( \bbr^3 ) \cap L^1_{ - 1 } ( \bbr^3 ) ) $ satisfying
\begin{align}
\sup_{ t \in [ 0 , T ] } ( \| f ( t ) \|_{ L^1_{ - 2 } } + \| f ( t ) \|_{ L^\infty_w } ) \leq C , \label{thm soft w}
\end{align}
where $ C > 0 $ depends on $ T $.
\end{thm}
\begin{proof}
In this theorem we will show that the sequence $ f_n $ converges in $ L^1 \cap L^1_{ - 1 } $. We follow the computation in the proof of Theorem \ref{thm hard} to obtain
\begin{align*}
& \frac{ d }{ d t } \| f_k - f_n \|_{ L^1_r } \leq J_1 + J_2 + J_3 + J_4 + J_5 ,
\end{align*}
where
\begin{align*}
J_1 & = \frac12 \iiint \frac{ \varrho^{ 2 + \gamma } }{ p^0 q^0 } ( f_k + f_n ) ( p ) | f_k - f_n | ( q ) ( p'^0 )^r d \omega \, d^3 q \, d^3 p , \\
J_2 & = \frac12 \iiint \frac{ \varrho^{ 2 + \gamma } }{ p^0 q^0 } ( f_k + f_n ) ( p ) | f_k - f_n | ( q ) ( q'^0 )^r d \omega \, d^3 q \, d^3 p , \\
J_3 & = \frac12 \iiint \frac{ \varrho^{ 2 + \gamma } }{ p^0 q^0 } ( f_k + f_n ) ( p ) | f_k - f_n | ( q ) ( p^0 )^r d \omega \, d^3 q \, d^3 p , \\
J_4 & = \iiint \1_{ \{ \varrho \leq k^{ - 1 } \} } \frac{ \varrho^{ 2 + \gamma } }{ p^0 q^0 } f_n ( p ) f_n ( q ) ( p'^0 )^r d \omega \, d^3 q \, d^3 p , \nonumber \\
J_5 & = \iiint \1_{ \{ \varrho \leq k^{ - 1 } \} } \frac{ \varrho^{ 2 + \gamma } }{ p^0 q^0 } f_n ( p ) f_n ( q ) ( p^0 )^r d \omega \, d^3 q \, d^3 p , \nonumber
\end{align*}
for $ r = 0 , - 1 $. We first choose $ - 1 < \delta < 0 $ satisfying
\begin{align*}
- 2 < \gamma < \delta < 0 , \qquad - 2 < \gamma + \delta < 0 ,
\end{align*}
and consider Young's inequality:
\begin{align}
( p^0 + q^0 )^{ 1 + \delta } \geq c ( p^0 )^{ \frac{ - \gamma + \delta }{ 2 } } ( q^0 )^{ 1 + \frac{ \gamma + \delta }{ 2 } } . \label{Young}
\end{align}
Then, $ J_1 $ is estimated as follows. For $ r = 0 $, we have
\begin{align*}
J_1 & \leq C \iint ( p^0 q^0 )^{ \frac{ \gamma }{ 2 } } ( f_k + f_n ) ( p ) | f_k - f_n | ( q ) \, d^3 q \, d^3 p \\
& \leq C \| f_k - f_n \|_{ L^1_{ \frac{ \gamma }{ 2 } } } \\
& \leq C ( \| f_k - f_n \|_{ L^1 } + \| f_k - f_n \|_{ L^1_{ - 1 } } ) ,
\end{align*}
since $ - 2 < \gamma < 0 $. For $ r = - 1 $, we apply Lemma \ref{lem 1/p'} and use Young's inequality \eqref{Young} to obtain
\begin{align*}
J_1 & \leq C \iint \frac{ \varrho^{ 2 + \gamma + \delta } }{ p^0 q^0 ( p^0 + q^0 )^{ 1 + \delta } } ( f_k + f_n ) ( p ) | f_k - f_n | ( q ) \, d^3 q \, d^3 p , \\
& \leq C \iint \frac{ \varrho^{ 2 + \gamma + \delta } }{ p^0 q^0 ( p^0 )^{ \frac{ - \gamma + \delta }{ 2 } } ( q^0 )^{ 1 + \frac{ \gamma + \delta }{ 2 } } } ( f_k + f_n ) ( p ) | f_k - f_n | ( q ) \, d^3 q \, d^3 p , \\
& \leq C \iint \frac{ ( p^0 q^0 )^{ 1 + \frac{ \gamma + \delta }{ 2 } } }{ ( p^0 )^{ 1 + \frac{ - \gamma + \delta }{ 2 } } ( q^0 )^{ 2 + \frac{ \gamma + \delta }{ 2 } } } ( f_k + f_n ) ( p ) | f_k - f_n | ( q ) \, d^3 q \, d^3 p , \\
& \leq C \iint \frac{ ( p^0 )^{ \gamma } }{ q^0 } ( f_k + f_n ) ( p ) | f_k - f_n | ( q ) \, d^3 q \, d^3 p , \\
& \leq C \| f_k + f_n \|_{ L^1_\gamma } \| f_k - f_n \|_{ L^1_{ - 1 } } \\
& \leq C \| f_k - f_n \|_{ L^1_{ - 1 } } ,
\end{align*}
since $ - 2 < \gamma < 0 $. The estimate of $ J_2 $ is the same as $ J_1 $:
\begin{align*}
J_2 \leq C ( \| f_k - f_n \|_{ L^1 } + \| f_k - f_n \|_{ L^1_{ - 1 } } ) ,
\end{align*}
for $ r = 0 , - 1 $. The estimate of $ J_3 $ for $ r = 0 $ is the same as $ J_1 $:
\begin{align*}
J_3 \leq C ( \| f_k - f_n \|_{ L^1 } + \| f_k - f_n \|_{ L^1_{ - 1 } } ) ,
\end{align*}
but for $ r = - 1 $, we estimate as follows:
\begin{align*}
J_3 & \leq C \iint \frac{ \varrho^{ 2 + \gamma } }{ ( p^0 )^2 q^0 } ( f_k + f_n ) ( p ) | f_k - f_n | ( q ) \, d^3 q \, d^3 p \\
& \leq C \iint \frac{ ( q^0 )^{ \frac{ \gamma }{ 2 } } }{ ( p^0 )^{ 1 - \frac{ \gamma }{ 2 } } } ( f_k + f_n ) ( p ) | f_k - f_n | ( q ) \, d^3 q \, d^3 p \\
& \leq C \| f_k + f_n \|_{ L^1_{ - 1 + \frac{ \gamma }{ 2 } } } \| f_k - f_n \|_{ L^1_{ \frac{ \gamma }{ 2 } } } \\
& \leq C ( \| f_k - f_n \|_{ L^1 } + \| f_k - f_n \|_{ L^1_{ - 1 } } ) .
\end{align*}
For $ J_4 $, we obtain for $ r = 0 $,
\begin{align*}
J_4 & \leq C \iint \1_{ \{ \varrho \leq k^{ - 1 } \} } \frac{ \varrho^{ 2 + \gamma } }{ p^0 q^0 } f_n ( p ) f_n ( q ) \, d^3 q \, d^3 p \\
& \leq C k^{ - 2 - \gamma } \| f_n \|_{ L^1_{ - 1 } }^2 .
\end{align*}
For $ r = - 1 $, we follow the above computation of the estimate of $ J_1 $ to obtain
\begin{align*}
J_4 & \leq C \iint \1_{ \{ \varrho \leq k^{ - 1 } \} } \frac{ \varrho^{ 2 + \gamma + \delta } }{ ( p^0 )^{ 1 + \frac{ - \gamma + \delta }{ 2 } } ( q^0 )^{ 2 + \frac{ \gamma + \delta }{ 2 } } } f_n ( p ) f_n ( q ) \, d^3 q \, d^3 p \\
& \leq C k^{ - 2 - \gamma - \delta } \iint \frac{ 1 }{ ( p^0 )^{ 1 + \frac{ - \gamma + \delta }{ 2 } } ( q^0 )^{ 2 + \frac{ \gamma + \delta }{ 2 } } } f_n ( p ) f_n ( q ) \, d^3 q \, d^3 p \\
& \leq C k^{ - 2 - \gamma - \delta } \| f_n \|_{ L^1_{ - 1 - \frac{ - \gamma + \delta }{ 2 } } } \| f_n \|_{ L^1_{ - 2 - \frac{ \gamma + \delta }{ 2 } } } \\
& \leq C k^{ - 2 - \gamma - \delta } ,
\end{align*}
where the weighted norms of $ f_n $ are bounded by Lemma \ref{lem f_ns}, since
\begin{align*}
- 2 < - 1 + \frac{ \gamma }{ 2 } < - 1 + \frac{ \gamma - \delta }{ 2 } < - 1 , \\
- 2 < - 2 - \frac{ \gamma + \delta }{ 2 } < - 1 .
\end{align*}
The estimate of $ J_5 $ with $ r = 0 $ is the same as $ J_4 $:
\begin{align*}
J_5 \leq C k^{ - 2 - \gamma } ,
\end{align*}
and for $ r = - 1 $, we have
\begin{align*}
J_5 & \leq C \iint \1_{ \{ \varrho \leq k^{ - 1 } \} } \frac{ \varrho^{ 2 + \gamma } }{ ( p^0 )^2 q^0 } f_n ( p ) f_n ( q ) \, d^3 q \, d^3 p \\
& \leq C k^{ - 2 - \gamma } \| f_n \|_{ L^1_{ - 2 } } \| f_n \|_{ L^1_{ - 1 } } .
\end{align*}
We combine the above estimates and apply Lemma \ref{lem f_ns} to conclude that the sequence $ f_k $ converges and there exists a solution $ f $ in $ L^1 ( \bbr^3 ) \cap L^1_{ - 1 } ( \bbr^3 ) $ on a time interval $ [ 0 , T ] $. The property \eqref{thm soft w} can be obtained by following the calculation in the proof of Lemma \ref{lem f_ns}, and this completes the proof.
\end{proof}

\section*{Acknowledgements}
In the course of this work the authors have benefitted from the financial support of a number of institutions. We all acknowledge support from the Centre International de Rencontres Mathematiques - Marseille Luminy where we benefitted from the Research in Residence program (CIRM 3431). We also acknowledge financial support under the project PID2024-155175NB-I00 (Agencia Estatal de Investigación and Ministerio de Ciencia, Innovación y Universidades). P.T. acknowledges support from St John's College Oxford. This work was supported by the National Research Foundation of Korea (NRF) grant funded by the Korea government (MSIT) (No.\ RS-2024-00451692).

\bibliographystyle{abbrv}

\end{document}